\documentclass[twoside]{amsart}
\usepackage{amsfonts,amsthm}
\usepackage{enumerate}
\theoremstyle{plain}
\newtheorem*{thm A}{Theorem~A}
\newtheorem*{thm B}{Theorem~B}
\newtheorem*{thm C}{Theorem~C}
\newtheorem*{thm D}{Theorem~D}
\newtheorem*{thm E}{Theorem~E}
\newtheorem*{thm F}{Theorem~F}
\newtheorem*{Main Theorem}{Main Theorem}

\newtheorem*{pro A}{Proposition~A}
\newtheorem*{pro B}{Proposition~B}
\newtheorem*{lem A}{Lemma~A}
\newtheorem*{coro}{Corollary}

\newtheorem{theorem}{Theorem}[section]

\newtheorem{lemma}[theorem]{Lemma}

\theoremstyle{definition}

\def \EN{{\eta}_{\nu}}
\def \KN{{\xi}_{\nu}}
\def \PN{{\phi}_{\nu}}
\def \PNK{{\phi}_{\nu}{\xi}}
\def \PNP{{\phi}_{\nu}{\phi}}
\def \E{\eta}
\def \Ph{\phi}
\def \al{\alpha}
\def \N{\nabla}
\def \SN{{\sum}_{{\nu}=1}^3}
\def \GBo{G_2({\Bbb C}^{m+1})}
\def \GBt{G_2({\mathbb C}^{m+2})}
\def \gtw{\hat \nabla ^{(k)}}
\def \gtwl{\hat \L ^{(k)}}
\def \L{\mathcal{L}}
\def \x{\xi}
\def \Ko{{\xi}_1}
\def \Kt{{\xi}_2}
\def \Ks{{\xi}_3}
\def \Dc{\mathfrak D^{\bot}}
\def \D{\mathfrak D}

\def \PKo{{\phi}{\xi}_1}
\def \Eo{{\eta}_1}
\def \Et{{\eta}_2}
\def \Es{{\eta}_3}
\def \Po{{\phi}_1}

\def \PoK{{\phi}_1{\xi}}
\begin{document}

\title[GTW Reeb Lie derivative structure Jacobi operator]
{Real hypersurfaces in complex two-plane Grassmannians with
GTW Reeb Lie derivative structure Jacobi operator}
\vspace{0.2in}
\author[E. Pak, G.J. Kim \& Y.J. Suh]{Eunmi Pak, Gyu Jong Kim and Young Jin Suh}
\address{\newline
Eunmi Pak, Gyu Jong Kim and Young Jin Suh
\newline Department of Mathematics,
\newline Kyungpook National University,
\newline Daegu 702-701, Korea}
\email{empak@hanmail.net}
\email{hb2107@naver.com}
\email{yjsuh@knu.ac.kr}

\footnotetext[1]{{\it 2010 Mathematics Subject Classification} : Primary 53C40; Secondary 53C15.}

\footnotetext[2]{{\it Key words and phrases} : Real hypersurface, Complex two-plane Grassmannian, Hopf hypersurface, 
Generalized Tanaka-Webster connection, Structure Jacobi Operator, Generalized Tanaka-Webster Lie derivative.}

\thanks{* This work was supported by grant Proj. No. NRP-2012-R1A2A2A-01043023.}

\begin{abstract}
  Using generalized Tanaka-Webster connection, we considered a real hypersurface $M$ in a complex two-plane Grassmannian $G_2({\mathbb C}^{m+2})$ when the GTW Reeb Lie derivative of the structure Jacobi operator coincides with the Reeb Lie derivative. Next using the method of simultaneous diagonalization, we prove a complete classification for a real hypersurface in $G_2({\mathbb C}^{m+2})$ satisfying such a condition. In this case, we have proved that $M$ is an open part of a tube around a totally geodesic $G_2({\mathbb C}^{m+1})$ in $G_2({\mathbb C}^{m+2})$.
\end{abstract}

\maketitle

\section*{Introduction}
\setcounter{equation}{0}
\renewcommand{\theequation}{\arabic{equation}}
\vspace{0.13in}

For real hypersurfaces with parallel curvature tensor,  many
differential geometers studied in complex projective spaces or in
quaternionic projective spaces~(\cite{KPSS,PSS,PS}). Different point of view, it is attractive to
classify real hypersurfaces in complex two-plane Grassmannians with certain conditions. For example, there is some result about parallel structure Jacobi operator (For more detail, see ~\cite{JPS,JMPS}). It is natural to question about complex two-plane Grassmannians.
\par
\vskip 3pt
As an ambient space,  a complex two-plane Grassmannian  $G_{2}(\mathbb C ^{m+2})$  consists of all
complex two-dimensional linear subspaces in $\mathbb C^{m+2}$. This
Riemannian symmetric space is the unique compact irreducible
Riemannian manifold being equipped with both a K\"{a}hler structure
$J$ and a quaternionic K\"{a}hler structure $\mathfrak{J}$ not
containing $J$. Then, we could naturally consider two geometric conditions for hypersurfaces $M$ in $G_{2}(\mathbb C ^{m+2})$, namely, that a $1$-dimensional
distribution $[\xi]= \text{Span} \{\xi \}$ and a $3$-dimensional
distribution $\mathfrak {D}^{\bot} = \text{Span}\{\xi_{1},\xi_{2},
\xi_{3}\}$ are both invariant under the shape operator~$A$ of $M$~(\cite{BS1}), where the {\it Reeb} vector field  $\xi$ is defined by $\xi = -JN$, $N$ denotes a local unit normal vector field of $M$ in
$G_{2}(\mathbb C ^{m+2})$ and  the {\it almost contact 3-structure}
vector fields $\{\xi_{1},\xi_{2},\xi_{3}\}$ are defined by $\xi_{\nu} = - J_{\nu} N$ $(\nu=1, 2, 3$).

\par
\vskip 3pt

By using  the result in Alekseevskii~\cite{Ale}, Berndt and Suh \cite{BS1} proved the following result about space of $type ~(A)$(sentence about (A)) and $type~ (B)$(one about (B))\,:

\begin{thm A}
Let $M$ be a connected orientable real hypersurface in $G_{2}(\mathbb C
^{m+2})$, $m \geq 3$. Then both $[\xi]$ and $\mathfrak D ^{\bot}$
are invariant under the shape operator of $M$ if and only if
\begin{enumerate}[\rm(A)]
\item {$M$ is an open part of a tube around a totally geodesic
$G_2({\mathbb C}^{m+1})$ in $G_2({\mathbb C}^{m+2})$, or}
\item{$m$ is even, say $m = 2n$, and $M$ is an open part of a
tube around a totally geodesic ${\mathbb H}P^n$ in $G_2({\mathbb
C}^{m+2})$.}
\end{enumerate}
\end{thm A}

\par

\vskip 3pt

When we consider the Reeb vector field $\xi$ in the expression of
the curvature tensor $R$ for a real hypersurface $M$ in $G_2({\Bbb
C}^{m+2})$, the structure Jacobi operator $R_{\xi}$ can be defined
in such as
\begin{equation*}
R_{\xi}(X)=R(X,{\xi}){\xi},
\end{equation*}
for any tangent vector field $X$ on $M$.
\par
\vskip 6pt Using the structure Jacobi operator $R_{\xi}$, Jeong,
P\'erez and Suh~\cite{JPS} considered a notion of {\it parallel structure
Jacobi operator}, that is, $\N_X R_\xi =0$ for any vector field $X$
on $M$, and gave a non-existence theorem.
And the authors~\cite{JMPS} considered the general
notion of {\it ${\frak D}^{\bot}$-parallel structure Jacobi
operator} defined in such a way that $ {\nabla}_{{\xi}_i}R_{\xi}=0,~
i=1,2,3, $ which is weaker than the notion of parallel structure
Jacobi operator. They also gave a non-existence theorem.
\vskip 6pt
\par
By the way, the Reeb vector field $\xi$ is said to be {\it
Hopf} if it is invariant under the shape operator $A$. The one
dimensional foliation of $M$ by the integral manifolds of the Reeb
vector field $\xi$ is said to be the {\it Hopf foliation} of $M$. We
say that $M$ is a {\it Hopf hypersurface} in $\GBt$ if and only if the
Hopf foliation of $M$ is totally geodesic. By the formulas in
Section~\ref{section 1} it can be easily checked that $M$ is Hopf if
and only if the Reeb vector field $\xi$ is Hopf.

\par
\vskip 6pt

Now, instead of the Levi-Civita connection for real hypersurfaces in
K\"{a}hler manifolds, we consider another new connection named {\it
generalized Tanaka-Webster connection}~(in short, let us say the
{\it GTW connection}) ${\hat\nabla}^{(k)}$ for a non-zero
real number $k$ (\cite{Kon}). This new connection $\gtw$ can be
regarded as a natural extension of Tanno's generalized
Tanaka-Webster connection~$\hat \nabla$ for contact metric
manifolds. Actually, Tanno~\cite{Tanno} introduced the generalized Tanaka-Webster connection ~$\hat \nabla$ for contact
Riemannian manifolds  by using the canonical connection on a nondegenerate, integrable
CR manifold.
\par
\vskip 6pt

 On the other hand, the original {\it Tanaka-Webster
connection}~(\cite{Tanaka,Web}) is given as a unique affine
connection on a non-degenerate, pseudo-Hermitian $CR$ manifolds
 associated with the almost contact structure. In particular,
if a real hypersurface in a K\"{a}hler manifold satisfies $\phi A + A
\phi = 2k \phi$ ($k \neq 0$), then the g-Tanaka-Webster connection
$\gtw$ coincides with the Tanaka-Webster connection.
\par
\vskip 6pt
Related to GTW connection, due to Jeong, Pak and Suh (\cite{JPS1,JPS2}), the {\it GTW Lie derivative}  was defined by
\begin{equation}
\gtwl_{X}Y=\gtw_{X}Y-\gtw_{Y}X,
\end{equation}
where $\gtw_{X}Y=\N_{X}Y + g(\phi AX,Y)\xi -\eta(Y)\phi AX-k\eta(X)\phi Y,~ k\in {\mathbb R}\setminus\{0\}$.\\
\vskip 6pt
In this paper, using the GTW Lie derivative, we consider a condition that the {\it GTW Reeb Lie derivative of the structure Jacobi operator coincides with the Reeb Lie derivative}, that is,
\begin{equation}
(\gtwl_{\xi}R_{\xi})Y=(\L_{\xi}R_{\xi})Y,
\end{equation}
for any tangent vector field $Y$ in $M$.
Using above notion, we have a classification theorem as follows\,:
\begin{Main Theorem}
Let $M$ be a connected orientable Hopf hypersurface in a complex two-plane Grassmannian $\GBt$, $m\geq 3$.  If the GTW Reeb Lie derivative of the structure Jacobi operator coincides with the Reeb Lie derivative and the Reeb curvature is non-vanishing constant along the Reeb vector field, then $M$ is an open part of a tube around a totally geodesic $G_2({\mathbb C}^{m+1})$ in $G_2({\mathbb C}^{m+2})$.
\end{Main Theorem}
\vskip 6pt
As a corollary, we consider a condition stronger than the condition (2) as follows\,:
 $$(\gtwl_{X}R_{\xi})Y=(\L_{X}R_{\xi})Y$$ for any tangent vector fields $X,Y$ in $M$. Then we assert the following
\begin{coro}
There do not exist any connected orientable Hopf real
hypersurfaces in $\GBt$, $m{\ge 3}$, with $(\gtwl_{X}R_{\xi})Y=(\L_{X}R_{\xi})Y$ when the Reeb
curvature is constant along the direction of the Reeb vector field.
\end{coro}

\vskip 6pt
 In section 1, we introduce basic equations in relation to the structure Jacobi operator and prove the key lemmas which will be useful to proceed our main theorem.  In section 2, we give a complete proof of the main theorem and corollary, respectively. In this paper, we refer to \cite{Ale,BS1,BS2,JPS,LS} for Riemannian geometric structures of $\GBt$ and its geometric quantities, respectively.
\vspace{0.15in}

\section{Key Lemmas}\label{section 1}
\setcounter{equation}{0}
\renewcommand{\theequation}{1.\arabic{equation}}
\vspace{0.13in}
\par
In this section, we introduce some fundamental equation of structure Jacobi operator and lemmas.
\begin{equation}\label{eq: 2.1}
\begin{split}
R_{\xi}X=& ~ R(X,\xi)\xi \\
=& ~X-\E(X)\x\\
&-\SN\Big\{\EN(X)\KN-\E(X)\EN(\x)\KN+3g(\PN X,\x)\PNK+\EN(\x)\PNP X \Big\}\\
&+\al AX-\al^{2}\E(X)\x,
\end{split}
\end{equation}
for any tangent field X on M.
\par
In \cite{JPS1}, they defined the  GTW Lie derivative as follows:
\begin{equation*}
\gtwl_{X}Y=\gtw_{X}Y-\gtw_{Y}X,
\end{equation*}
where $\gtw_{X}Y=\N_{X}Y + F_X Y, ~F_X Y=g(\phi AX,Y)\xi -\eta(Y)\phi AX-k\eta(X)\phi Y$. The operator $F_X Y$ said to be the {\it generalized Tanaka-Webster operator}~(in short, GTW operator).
Putting $X=\x$ and $Y=\x$, the GTW operator is written as
\begin{equation}\label{eq: op}
F_\x Y=-k\phi Y ~\text {and}~ F_X \x=-\phi AX, ~\text{respectively}.
\end{equation}
For an (1-1) type tensor $R_{\xi}$,
this condition $(\gtwl_{X}R_{\xi})Y=(\L_{X}R_{\xi})Y$ is equivalent to
\begin{equation}\label{eq: 2.2}
F_{X}(R_{\xi}Y)-F_{R_{\xi}Y}X-R_{\xi}F_{X}Y+R_{\xi}F_{Y}X=0.
\end{equation}

Replacing X=$\x$ in \eqref{eq: 2.2}, we get
\begin{equation}\label{eq: 2.3}
-k\Ph R_{\x}Y+ \Ph AR_{\x}Y+kR_{\x}\Ph Y-R_{\x}\Ph AY=0.
\end{equation}

Since $R_{\x}$ is a symmetric tensor field, taking symmetric part of \eqref{eq: 2.3}, we have
\begin{equation}\label{eq: 2.4}
kR_{\x}\Ph Y-R_{\x}A\Ph Y-k\Ph R_{\x}Y+A\Ph R_{\x}Y=0.
\end{equation}

Subtracting \eqref{eq: 2.4} from \eqref{eq: 2.3}, we obtain
\begin{equation}\label{eq: 2.7}
(\Ph A-A\Ph)R_{\xi}Y=R_{\xi}(\Ph A-A\Ph)Y.
\end{equation}
Therefore, this condition that the GTW Reeb Lie derivative of the structure Jacobi operator coincides with the Reeb Lie derivative has such a geometric condition, that is, $(\Ph A-A\Ph)$ and $R_{\xi}$ commute with each other.

Putting $Y=\x$ in \eqref{eq: 2.2} and using \eqref{eq: op}, \eqref{eq: 2.2} is replaced by
\begin{equation}\label{eq: 2.6}
R_{\x}(\Ph AX)-kR_{\x}(\Ph X)=0.
\end{equation}
Taking the transpose part on \eqref{eq: 2.6}
,
we get
\begin{equation}\label{eq: 2.7}
-A\Ph R_{\x}X+k\Ph R_{\x}X=0.
\end{equation}

By using above these equations, we can give two lemmas which contribute to prove our main theorem.
\vskip 5pt
\begin{lemma}\label{lemma 3.1}
Let M be a Hopf hypersurface $M$ in $\GBt$. If the GTW Reeb Lie derivative of the structure Jacobi operator coincides with the Reeb Lie derivative of this operator and the principal curvature $\alpha$ is constant along the direction of the Reeb vector field $\xi$, then the Reeb vector field $\x$ belongs to the distribution $\D$ or the distribution $\Dc$
\end{lemma}
\begin{proof}
Let us put $\x=\E(X_{0})X_{0}+\Eo(\Ko)\Ko,$ for some unit vector fields $X_{0} \in \D$ and $\Ko \in \Dc$.
If $\al=0$, then $\x \in \D$ or $\x \in \Dc$, which is proved by P\'erez and Suh (\cite{PS1}).\\
So, we consider the other case $\al \neq 0$.\\
Putting $X=\Ko$ into \eqref{eq: 2.1} and using $A\Ko=\al\Ko$, we have
\begin{equation}\label{eq: 3.1}
R_{\x}(\Ko)=\al^{2}\Ko-\al^{2}\E(\Ko)\x.
\end{equation}
Replacing $X=\PKo$ into \eqref{eq: 2.1},  \eqref{eq: 2.1} becomes
\begin{equation}\label{eq: 3.2}
R_{\x}(\PKo)=(\al^{2}+8\E^{2}(X_{0}))\Po\x.
\end{equation}
Putting $X=\x$ into \eqref{eq: 2.2} and using \eqref{eq: op}, \eqref{eq: 2.1} is written as
\begin{equation}\label{eq: 3.4}
-k\Ph R_{\x}Y+\Ph AR_{\x}Y+kR_{\x}(\Ph Y)-R_{\x}(\Ph AY)=0.
\end{equation}
Substituting $Y=\Ko$ in the above equation and using \eqref{eq: 3.1}, \eqref{eq: 3.2}, it becomes
\begin{equation}\label{eq: 3.6}
8(k-\al)\E^{2}(X_{0})\PoK=0.
\end{equation}
Taking the inner product with $\PoK$, we get
\begin{equation}
8(k-\al)\E^{4}(X_{0})=0.
\end{equation}
This equation induces that $k=\al$ or $\E^{4}(X_{0})=0$. Therefore, it completes the proof of our Lemma.
\end{proof}
In next section, we will give a complete proof of our main theorem. In
order to do this, first we consider the case that
$\x \in \Dc$. Without loss of generosity, we may put $\x=\x_1$.

\begin{lemma}\label{lemma 3.2}
Let M be a Hopf hypersurface in $\GBt$ when the Reeb curvature is non-vanishing. If the GTW Reeb Lie derivative of the structure Jacobi operator coincides with the Reeb Lie derivative of this operator and the Reeb vector field $\x$ is belong to the distribution $\Dc$,
then the shape operator $A$ commutes with the structure tensor $\Ph$.
\end{lemma}
\begin{proof}
Putting $\x=\Ko$ in \eqref{eq: 2.1},
we get
\begin{equation}\label{eq: 3.8}
R_{\x}X=X-\E(X)\x-\Po \Ph X+\al AX-\al^{2}\E(X)\x+2\Et(X)\Kt+2\Es(X)A\Ks.
\end{equation}
Replacing $X$ with $AX$ in \eqref{eq: 3.8},  it is written as
\begin{equation}\label{eq: 3.9}
R_{\x}AX=AX-\al\E(X)\x-\Po \Ph AX+\al A^{2}X-\al^{3}\E(X)\x+2\Et(AX)\Kt+2\Es(AX)A\Ks.
\end{equation}
And applying the shape operator $A$ on  \eqref{eq: 3.8}, \eqref{eq: 3.8} becomes
\begin{equation}\label{eq: 3.10}
AR_{\x}X=AX-\al\E(X)\x-A\Po \Ph X+\al A^{2}X-\al^{3}\E(X)\x+2\Et(X)A\Kt+2\Es(X)A\Ks.
\end{equation}
On the other hand, applying the structure tensor field $\Ph$ to the equation (1.8) in \cite{LSW}, we get
\begin{equation}\label{eq: 5.11}
AX=\alpha \E(X) \x +2\E_2(AX) \x_2+2\E_3(AX) \x_3-\Ph \Ph_1 AX.
\end{equation}
Taking the symmetric part of \eqref{eq: 5.11}, we obtain
\begin{equation}\label{eq: 5.12}
AX=\alpha \E(X) \x +2\E_2(X)A \x_2+2\E_3(X)A \x_3-A \Ph_1 \Ph X.
\end{equation}
Putting $\nu =1$ in the first equation of (1.5) in \cite{JPS1}, it becomes
\begin{equation}\label{eq: 5.13}
 \Ph \Ph_1 X=  \Ph_1 \Ph X.
\end{equation}
Using \eqref{eq: 5.11}, \eqref{eq: 5.12} and subtracting \eqref{eq: 3.10} from \eqref{eq: 3.9}, we have
\begin{equation}\label{eq: 3.11}
R_{\x}AX=AR_{\x}X.
\end{equation}
From \eqref{eq: 3.11} and putting $Y=X$,  \eqref{eq: 2.7} is written as
\begin{equation}\label{eq: 3.12}
A(R_{\xi} \Ph-\Ph R_{\xi})X=(R_{\xi}\Ph -\Ph R_{\xi})AX.
\end{equation}
 Putting $X=\Ph X$ in \eqref{eq: 3.8}, we have
 \begin{equation}\label{eq: 5.14}
R_{\x}\Ph X=\Ph X-\Po \Ph^2 X+\al A\Ph X+2\Et(\Ph X)\Kt+2\Es(\Ph X)A\Ks.
\end{equation}
   Applying the structure tensor field $\Ph$ to \eqref{eq: 3.8}, we get
 \begin{equation}\label{eq: 5.15}
\Ph R_{\x}X=\Ph X-\Ph \Po \Ph X+\al \Ph AX+2\Et(X)\Ph \Kt+2\Es(X)A\Ph \Ks.
\end{equation}
 Subtracting \eqref{eq: 5.15} from \eqref{eq: 5.14}, we obtain
  \begin{equation}\label{eq: 5.16}
(R_{\x}\Ph-\Ph R_{\x})X=\al ( A\Ph -\Ph A)X.
\end{equation}
Using the equation \eqref{eq: 5.16}, the equivalent condition of \eqref{eq: 3.12} is this one as
\begin{equation}
\al A(A\Ph-\Ph A)X=\al (A\Ph-\Ph A)AX.
\end{equation}
\vskip 10pt
\par
By our assumption $\al \neq 0$, the above equation can be replaced by
\begin{equation}\label{eq: 3.13}
A(A\Ph-\Ph A)X=(A\Ph-\Ph A)AX.
\end{equation}
Because of \eqref{eq: 3.13}, there is a common basis $ \{e_{i}\mid i=1,..., 4m-1\} $  such that
\begin{equation}\label{eq: 3.14}
Ae_{i}=\lambda_{i}e_{i}
\end{equation}
and
\begin{equation}\label{eq: 3.15}
(A\Ph-\Ph A)e_{i}=\gamma_{i}e_{i}.
\end{equation}
Using \eqref{eq: 3.14}, \eqref{eq: 3.15} becomes
\begin{equation}\label{eq: 3.16}
\gamma_{i}e_{i}=A\Ph e_{i}-\Ph Ae_{i}=A\Ph e_{i}-\lambda_{i}\Ph e_{i}.
\end{equation}
Taking the inner product \ with $e_{i}$, we get $\gamma_{i}=0$.\\
Since the eigenvalue $\gamma_{i}$ vanishes  for all $i$, from \eqref{eq: 3.15} we conclude that
\begin{equation}
 A\Ph-\Ph A=0.
\end{equation}
Consequently, we proved this lemma.
\end{proof}

\section{Proof of the main theorem}\label{section 2}
\setcounter{equation}{0}
\renewcommand{\theequation}{2.\arabic{equation}}
\vspace{0.13in}
Let us consider a Hopf hypersurface $M$ in $G_2({\mathbb C}^{m+2})$ with $(\gtwl_{\xi}R_{\xi})Y=(\L_{\xi}R_{\xi})Y$.\\
By Lemma 1 in section 1, we can conclude that the Reeb vector field $\x$ in $M$ belongs either to the distribution $\D$ or $\Dc$.\\
Then, we can devide the following two cases:\\
$\bullet$ Case I: $\x \in\Dc$\\
$\bullet$ Case II:$\x \in\D$\\
Now, we check the first case in our consideration.\\
\vskip 10pt
If $\x\in\Dc$, by Theorem A and Lemma 2, we can assert that M is locally congruent to the model space of type (A).
We have to check if the model space of type (A) satisfies the condition $(\gtwl_{\xi}R_{\xi})Y=(\L_{\xi}R_{\xi})Y$ or not.
For type (A)-space, detail information (eigenspaces, corresponding eigenvalues, and multiplicities) was given in \cite{BS1}.

Putting $X=\x$ in \eqref{eq: 2.2}, we get the equivalent condition of $(\gtwl_{\x}R_{\xi})Y=(\L_{\x}R_{\xi})Y$ as follows\,:
\begin{equation}\label{eq: 4.1}
-k\Ph R_{\x}Y+\Ph A R_{\x}Y+ kR_{\x}\Ph Y-R_{\x}\Ph AY=0.
\end{equation}
On the other hand, putting $\xi=\Ko$ into \eqref{eq: 2.1}, we get
\begin{equation}\label{eq: 4.6}
R_{\xi}X=X-\E(X)\xi-\Po\Ph X+\al AX-\al^{2}\E(X)\xi+2\Et(X)\Kt+2\Es(X)\Ks.
\end{equation}
Using \eqref{eq: 4.1} and \eqref{eq: 4.6}, we get the following result\,:
\begin{equation}\label{eq: 4.2}
-k\Ph(R_{\x}Y)+\Ph A(R_{\x}Y)+R_{\x}k\Ph Y-R_{\x}\Ph AY  = \left\{ \begin{array}{ll}
                0,                    & \mbox{if}\ \  Y \in T_{\al}\\
                0,                    & \mbox{if}\ \  Y \in T_{\beta} \\
                0,                    & \mbox{if}\ \  Y \in T_{\lambda}\\
                0,                    & \mbox{if}\ \  Y \in T_{\mu}.\\
\end{array}\right.
\end{equation}
Therefore, we can assert that if $\x$ in $\Dc$, then $M$ is an open part of a tube around a totally geodesic $\GBo$ in $\GBt$.

\vskip 10pt

If the Reeb vector field $\x \in \D$, due to \cite{LS}, we can assert that M is locally congruent to space of type (B).
It remains whether type (B)-space satisfies this condition $(\gtwl_{X}R_{\xi})Y=(\L_{X}R_{\xi})Y$.
Also, by using information  of type (B)-space given in \cite{BS1}, we can check this problem.

We suppose that type (B)-space satisfies $(\gtwl_{\x}R_{\xi})Y=(\L_{\x}R_{\xi})Y$.
Then, as an equivalent condition, this space must satisfy
\begin{equation}\label{eq: 4.3}
-k\Ph(R_{\x}Y)+\Ph A(R_{\x}Y)+R_{\x}k\Ph Y-R_{\x}\Ph AY=0.
\end{equation}
Since $\xi$ is belong to $\mathfrak{D}$, the structure Jacobi operator in $\GBt$ can be replaced as follows:
\begin{equation}\label{eq: 4.4}
R_{\xi}X=X-\E(X)\x-\SN \Big\{\EN(X)\KN+3g(\PN X,\x)\PNK \Big\}+\al AX-\al^{2}\E(X)\x.
\end{equation}
Applying $Y=\Po \xi \in T_{\gamma}$ into \eqref{eq: 4.3} and using \eqref{eq: 4.4}, we get
\begin{equation}
k(4-\al\beta)\Ko=0.
\end{equation}
Since $k\neq0$ and $\al \beta=4$, this makes a contradiction.

Hence summing up these assertions, we have given a complete proof of our main theorem in the introduction. \hspace{7.7cm} $\Box$

\section{Proof of Corollary}\label{section 3}
\setcounter{equation}{0}
\renewcommand{\theequation}{3.\arabic{equation}}
\vspace{0.13in}

In this section, we consider another problem for this condition
\begin{equation}\label{eq: **}
(\gtwl_{X}R_{\xi})Y=(\L_{X}R_{\xi})Y,
\end{equation}
 for any tangent vector fields $X,Y$ in $M$.\\
 \par
If the Reeb curvature is non-vanishing, the condition $\Ph A=A \Ph$ have already proved in Lemma 1.2. Thus, we now consider only the case that $\al$ is vanishing. Under these
assumptions, we give the following lemma.
\begin{lemma}\label{lemma 4.1}
Let M be a Hopf hypersurface in $\GBt$ with vanishing the Reeb curvature. If the GTW Reeb Lie derivative of structure Jacobi operator coincides with Reeb Lie derivative of this operator and the Reeb vector field $\x$ is belong to the distribution $\Dc$,
then shape operator A and the structure tensor $\Ph$ commute each other.
\end{lemma}

\begin{proof}
Recall that \eqref{eq: 2.2} was given by
\begin{equation}\label{eq: 3.17}
F_{X}(R_{\x}Y)-F_{R_{\x}Y}X-R_{\x}F_{X}Y+R_{\x}F_{Y}X=0.
\end{equation}
Putting $X=\xi$ in the above equation and using \eqref{eq: 2.6}, \eqref{eq: 2.7}, \eqref{eq: 3.17} is written as
\begin{equation}\label{eq: 3.18}
(\Ph A-A \Ph)R_{\xi}Y=0.
\end{equation}
Applying $\al=0$ in \eqref{eq: 4.6}, it becomes
\begin{equation}\label{eq: 3.19}
R_{\x}X=X-\E(X)\x-\Po \Ph X+2\Et(X)\Kt+2\Es(X)A\Ks.
\end{equation}
On the other hand, applying  $\Ph$ and $X=\Ph X$ to \eqref{eq: 5.12}, respectively, we have
\begin{equation}\label{eq: 3.20}
\begin{split}
& \Ph AX=2\E_2(X)\Ph A \x_2+2\E_3(X)\Ph A \x_3-\Ph A \Ph_1 \Ph X,\\
& A\Ph X=2\E_3(X) A \x_2-2\E_2(X) A \x_3- A \Ph_1 \Ph^2  X.
\end{split}
\end{equation}
Combining \eqref{eq: 3.18}, \eqref{eq: 3.19}, \eqref{eq: 3.20} and using \eqref{eq: 5.13}, we get
\begin{equation}
2(\Ph A-A\Ph)Y=0.
\end{equation}
Therefore we also get the same conclusion in case of $\al=0$.
\end{proof}

By Lemmas 1.2 and 3.1, we can assert that if $\xi \in \Dc$, then $M$ is the model space of type (A).
Now we need to check if the space of type (A) satisfies \eqref{eq: **} or not.

Then the type (A)-space must satisfy the following condition
\begin{equation}\label{eq: 5.1}
F_{X}R_{\xi}Y-F_{R_{\xi}Y}X-R_{\xi}F_{X}Y+R_{\xi}F_{Y}X=0.
\end{equation}
Putting $Y=\xi$ into \eqref{eq: 5.1}, we have
\begin{equation}\label{eq: 5.2}
R_{\xi}\Ph AX-kR_{\xi}\Ph X=0.
\end{equation}
By using \eqref{eq: 3.19}, \eqref{eq: 5.2} becomes
\begin{equation}\label{eq: 5.3}
\begin{split}
\Ph AX&+\Po AX+\al A\Ph AX+2\Es(AX)\Kt-2\Et(AX)\Ks\\
-k\Ph X &-k\Po X-k\al A\Ph X-2k\Es(X)\Kt+2k\Et(X)\Ks=0.
\end{split}
\end{equation}
Replacing $\Kt$ into X, we get
\begin{equation}
(\al \beta+2)(k-\beta)\Ks=0.
\end{equation}
Taking the inner product with $\Ks$, the above equation implies $\al \beta=-2$ or $k=\beta$.
However, since $k\neq 0,~ \alpha = \sqrt{8}\cot(\sqrt{8}r)$ and $\beta =\sqrt{2}\cot(\sqrt{2}r)$, this makes a contradiction.

Hence  we can assert our corollary in the introduction. \hspace{3.4cm} $\Box$

\vspace{0.15in}

\end{document}